%% file: ms.tex
\documentclass[11pt]{amsart}

\pdfoutput=1 

\usepackage{style}
\include{config}

\include{macros}

\begin{document}

\begin{abstract}
We consider the endomorphism operad of a functor, which is roughly the
object of natural transformations from (monoidal) powers of that functor
to itself.

There are many examples from geometry, topology, and algebra where this
object has already been implicitly studied. 

We ask whether the endomorphism operad of the forgetful functor from
algebras over an operad to the ground category recovers that operad.
The answer is positive for operads in vector spaces over an infinite
field, but negative both in vector spaces over finite fields and in
sets.

Several examples are computed.
\end{abstract}
\maketitle
\tableofcontents

\section{Introduction}

It is commonplace in contemporary mathematics to discuss the natural 
endomorphisms of a functor. Cohomology operations, the center of
a category, and reconstruction theorems in the style of Tannaka
are examples of this line of study. 

The raison d'être of this article is to discuss the natural 
endomorphisms in a more general sense. Specifically we study the 
natural endomorphism \emph{operad} of a functor. As we argue below,
this generalization is not at all exotic---mathematicians were 
studying endomorphism operads of functors decades before operads or 
functors were explicitly defined. 
And there are multiple natural questions of contemporary interest
that are best described as the study of the endomorphism operad of
a particular functor.

However, this point of view seems to have remained mostly implicit
in the literature (one notable exception is
\emph{Modules over Operads and Functors}%
~\cite[\S\thinspace 3.4]{doi:10.1007/978-3-540-89056-0}, where the 
subject is treated cleanly but in passing).

Our contribution here is 
\begin{enumerate}
	\item to give some outlines of the structure of the operadic theory, 
	      exploiting the fact that the endomorphism operad of a functor is
	      adjoint to
	      \[
	      	P\longmapsto \alg P,
	      \] 
	\item to discuss operadic approximations to functors, and in
	      particular successes and failures of \emph{reconstruction} in
	      the operadic context, and
	\item to calculate some interesting and illustrative examples.
\end{enumerate}

Let us discuss the structure of the article.

After the introduction, in which we define our terms, we describe some
historical examples of interest as motivation%
~[\sectionref{section:
What endomorphism operads of functors are about}].
Then we discuss the interpretation of the endomorphism operad of a
functor as a universal operadic approximation to that functor%
~[\sectionref{section: approximation}].
After that, we turn to the question of reconstruction%
~[\sectionref{section: reconstruction}]. This is the question of whether
and when the endomorphism construction, applied to the forgetful functor
from algebras over an operad to the ground category, recovers that
operad.

We defer all proofs to the following three sections, where we introduce
some technical tools and apply them to justify the statements and 
examples of the previous sections.

For our conventions, we fix a presentable closed symmetric monoidal
category $(\ground, \otimes, \monu)$ and denote by
\[
	\begin{tikzcd}
		\ground\op \times \ground \ar[rr, "{[-,-]}"]&& \ground
	\end{tikzcd}
\]
its associated self-enrichment.
We let $\monoids{\ground}$ denote the category
of monoids in $\ground$ and $\operads{\ground}$ denote the category
of operads in $\ground$, i.e., monoids in the monoidal category of
symmetric sequences in $\ground$.

\subsection{Representations of monoids and operads}

Given a monoid $M$ in $\monoids{\ground}$ we denote by $\rep M$ the
category of $M$-representations in $\ground$. The category $\rep M$
comes equipped with the forgetful functor $\forget M$ to $\ground$.

Similarly, given an operad $P$ in $\operads{\ground}$, we denote by
$\alg P$ the category of $P$-algebras in $\ground$. The category
$\alg P$ comes equipped with the forgetful functor $\forget P$ to
$\ground$.

In both cases this assignment is functorial. A map $f:M\to N$
of monoids induces a functor $\rep N\to \rep M$ over $\ground$,
and a map $f:P\to Q$ of operads induces a functor
$\alg Q\to \alg P$ over $\ground$.

The conditions we have set on $\ground$ ensure that the categories of
representations and all induced and forgetful functors are accessible
 and so we get functors
\begin{align*}
\repfunctor:\monoids{\ground}
&\longrightarrow \left(\slice{\accesscats} \ground\right)\op
\\
\algfunctor:\operads{\ground}
&\longrightarrow \left(\slice{\accesscats} \ground\right)\op
\end{align*}

\subsection{Endomorphisms of functors to \texorpdfstring{$\ground$}{C}}

Given a functor $F$ with codomain $\ground$, we may take the
endomorphism monoid $\automonoid F$ or the endomorphism operad
$\autooperad F$ of natural transformations from $F$ to itself. The
monoidal case is classical. The operadic case occurs in more recent
literature%
~\cite[3.4]{doi:10.1007/978-3-540-89056-0}.

Recall that given two functors $F, G : \domain \to \ground$,
up to size issues which can be finessed in a variety of ways,
the $\ground$-natural transformations from $F$ to $G$ are presented
by the object
\[
	\nat \ground F G \coloneqq \coint {\domain}
	[F -, G -].
\]
Here, following Yoneda's original notation%
~\cite[\S\thinspace 4]{On_Ext_and_exact_sequences},
$\coint {\domain}$ denotes the cointegration (or end)
of a functor $\domain\op \times \domain \to \ground$.

\begin{definition}%
\label{definition: endo operad and monoid}
Let $F : \domain \to \ground$ be an accessible functor.
The \emph{endomorphism monoid} of $F$ is
\[
	\automonoid F \coloneqq \nat \ground F F
\]
the monoid of $\ground$-natural transformations of $F$.

The \emph{endomorphism operad} of $F : \domain \to
\ground$ is, in arity $n$
\[
	\autooperad F(n) \coloneqq \nat \ground {F^{\otimes n}} F
\]
the operad of $\ground$-natural transformations of $F$.
\end{definition}

In both cases the structure operations are induced by composition in 
the closed category $\ground$. 
In particular, it follows from the definition that the endomorphism 
monoid $\automonoid F$ is the arity one component of the endomorphism 
operad $\autooperad F$.

Accessibility of $F$ and presentability of $\ground$ ensure that
the endomorphism monoid and operad are
small~\cite[3.1]{arXiv:1904.06987}.

For example, if $X:*\to\ground$ is an object of $\ground$, then
$\autooperad X$ is the ordinary endomorphism operad with components
$[X^{\otimes n},X]$.

\begin{remark}[Coendomorphism operad]
Dually, one can define the coendomorphism operad of a functor
$F: \domain \to \ground$ via
\[
	\autooperad F(n) \coloneqq \nat \ground F {F^{\otimes n}}.
\]
This notion has also been actively studied in recent literature%
~\cite{doi:10.24033/ast.904}.
\end{remark}

\begin{remark}
Unfortunately, the name ``endomorphism operad of a functor'' has been
used in the literature for at least three different notions.
\begin{enumerate}
	\item Here the ``endomorphism operad of a functor'' $F :
	      \domain\to\mathcal{C}$ (with $\mathcal{C}$ symmetric monoidal)
	      is the endomorphism operad of the object $F$ in the symmetric
	      monoidal category of functors from $\domain$ to $\mathcal{C}$;
	\item in unpublished work, May used the term ``endomorphism operad 
	of a functor'' $F$ for what we have called the coendomorphism operad 
	of $F$;
	\item Richter and others have defined another kind of ``endomorphism 
		operad of a functor'' $F: \domain\to \mathcal{C}$ where both 
		$\domain$ and $\mathcal{C}$ are monoidal in their study of 
		transfer of algebraic structure%
~\cite{doi:10.1016/j.jpaa.2005.04.017,isbn:978-0-8218-4776-3}. 
\end{enumerate}
\end{remark}

Both the endomorphism monoid and endomorphism operad are adjoint 
to categories of representations.

\begin{proposition}[{\cite[4.1]{arXiv:1904.06987}}]
We have the following adjunctions:
\[
	\begin{tikzcd}[ampersand replacement=\&]
		\monoids{\ground} \arrow[rr, shift left=2,"\repfunctor"]
		\& \&
		\left(\slice{\accesscats} \ground\right)\op
		\arrow[ll, shift left=2, "\automonoidfunctor"]
		\\
		\operads{\ground} \arrow[rr, shift left=2,"\algfunctor"]
		\& \&
		\left(\slice{\accesscats} \ground\right)\op.
		\arrow[ll, shift left=2, "\autooperadfunctor"]
	\end{tikzcd}
\]
\end{proposition}

In some of the motivating examples below we consider endomorphism
monoids and operads of functors whose domains are not accessible. This
requires a variation of the setup%
~\cite[2.2, 2.4]{arXiv:1904.06987}
but in practice does not affect the theory very much. 

\section{A historical perspective}%
\label{section: What endomorphism operads of functors are about}
We discuss some motivation and historical antecedents to this theory.

\subsection{Lie groups and Lie algebras}%
\label{subsec:Lie groups}

One of the earliest examples of algebraic structures stemming from
endomorphisms of functors comes from Lie groups and Lie algebras.

Consider the functor $\tangentate$ from real Lie groups to real vector
spaces which takes the tangent space at the identity element. Since the
work of Lie and Klein on ``continuous groups'', it has been known that
this functor factorizes through the category of Lie algebras. In other
words, there is a morphism of operads
\[
	\Lie \longrightarrow \autooperad \tangentate.
\]
Unsurprisingly (to the modern reader), this operad 
map is an isomorphism. 

\begin{restatable}{proposition}{opoftangent}
Let
\[
	\tangentate : \Lie\Grp \longrightarrow \vect \reals
\]
be the functor sending a Lie group $G$ to its tangent space at the
unit. Then
\[
	\autooperad \tangentate = \Lie.
\]
\end{restatable}

\subsection{Cohomology operations}%
\label{subsec:cohomology}

In the 1950s, the early category theorists applied then relatively new
notions of naturality to study cohomology operations, the elements of
the natural endomorphism monoid of the cohomology functor%
~\cite{doi:10.1073/pnas.37.1.58,
doi:10.2307/1969820,doi:10.2307/1969702,doi:10.2307/1969849,
doi:10.1016/0001-8708(72)90004-7, 
doi:10.5169/seals-25343
}.

Cohomology also has natural multilinear operations, most notably the cup
product, and there are many examples where this multilinear structure
gives an easy obstruction to the existence of some topological map.
This points to the utility of a general framework with which to discuss
natural multi-ary operations on a functor.

\subsection{Manifolds}%
\label{subsec: manifolds}
In the context of the geometry of manifolds, there has been a thread of 
research on natural differential operators on various vector bundles 
associated functorially to manifolds. 
Possibly the earliest results along these lines in the literature are
due to Palais~\cite{doi:10.1090/s0002-9947-1959-0116352-7}. 
We can recast his results as a computation of the endomorphism monoid 
of the functor of differential forms
\[
	\begin{tikzcd}
		\upOmega^* : \mathsf{Diff}\op \longrightarrow \vect\reals
		\longrightarrow \Sets
	\end{tikzcd}
\]
In essence, he computes that the endomorphism monoid of $\upOmega^*$
is spanned by scalar multiplication and exterior differentiation.

Over the years these results have been generalized, both to other 
tensorial constructions and to multilinear operators~%
\cite[\S\thinspace 34]{doi:10.1007/978-3-662-02950-3}. For example, 
bilinear natural maps on tensor powers of the tangent and cotangent
bundle are fully classified. These contain mostly expected operations
like the wedge product, Lie bracket, and variations thereof but there is
an exceptional bilinear operator acting on certain spaces of densities
on $1$-manifolds.

We can ask a few simple questions intended to reformulate these 
questions of geometric naturality in our language.
Some of the answers may already be known to geometers.

\begin{question*} 
What is the endomorphism operad of differential forms? 
\end{question*}
The natural guess is that it contains only scalar multiplication, 
	exterior differentiation, and the wedge product, i.e., that
	the endomorphism operad, properly construed, is the operad governing
	commutative algebras equipped with a square zero derivation.
\begin{question*}
What is the endomorphism operad of differential forms 
	viewed as a functor with domain a category of manifolds with some
	additional structure?  For instance, what is the endomorphism operad 
	in the Riemannian setting?
\end{question*}
With the appropriate setup, this operad includes the natural operations
for smooth manifolds along with the adjoint $\mathrm{d}^*$ of the
exterior derivative and operations derived from $\mathrm{d}$,
$\mathrm{d}^*$, and the wedge product, such as the Laplacian and the
bracket. The operad governing this collection of structures with
only the ``obvious'' relations among them is the operad governing
Batalin--Vilkovisky algebras equipped with a square zero derivation of
the product with no assumed compatibility with the BV operator.

A version of this question can also be asked for manifolds equipped
with further structure: complex 
(a different point of view on this is given by Millès%
~\cite{arXiv:14093604M}),
Kähler, and so on.

\begin{question*} 
What is the endomorphism operad of various powers of the tangent bundle?
\end{question*}
To give a fully satisfactory answer to this suite of questions we would
want a two-colored operad dealing simultaneously with the tangent and
cotangent bundle. Formulating this structure precisely is not within our
purview because some of these bundles are covariant functors and some
contravariant.

\subsection{Hochschild-type examples}
Part of the literature on Hochschild 
(co)homology, cyclic (co)homology, and the bar complex%
~\cite{%
doi:10.2307/1969820,
doi:10.2307/1970343,
doi:10.1007/BF02698807,
doi:10.1007/978-3-662-21739-9_4,
doi:10.1007/BF01077036,
doi:10.1007/s002080050343, 
doi:10.1007/s10587-007-0074-4,
doi:10.4171/JNCG/10,
doi:10.5802/aif.2417,
doi:10.24033/bsmf.2576,
MR2818707,
isbn:978-2-85629-363-8,
doi:10.4171/jncg/167,
doi:10.1016/j.jalgebra.2015.09.018,
doi:10.4171/JNCG/304,
doi:10.1093/imrn/rny241,
arXiv:1903.01437
} (this is not exhaustive) has been
devoted to finding ever-more refined natural algebraic structures
on the various homology groups and chain complexes involved in the
constructions, sometimes on restricted classes of algebras. This thread
of work can be interpreted as providing partial or full computations of
the endomorphism operad of the functor under consideration in the given
context.

\subsection{Homotopical versions}
All of the statements in this paper are one-categorical and rigid.
It is reasonable (probably more reasonable) to ask about a homotopically
coherent version of this story.

We outline a speculative application. Any category $\ground$ can be 
viewed as enriched in sets, and then the identity functor of $\ground$ 
is the forgetful functor for representations of the trivial monoid. 
The endomorphism monoid in this setting is called the
\emph{center} of $\ground$.

There are various settings in enriched category theory in which 
the notion of a \emph{derived center} has been defined%
~\cite{doi:10.1016/j.aim.2012.04.011}.
Perhaps there is something interesting to say about the \emph{derived 
central operad}, defined along the lines of%
~\ref{definition: central operad}.

For example, it is known that the derived center of the category of 
simplicial algebras over a Lawvere theory, viewed as enriched in 
simplicial sets, is homotopy
equivalent to the ordinary center of discrete
algebras over the same theory%
~\cite{doi:10.1007/s00209-016-1744-4}.
But perhaps there is higher homotopical information in the derived 
central operads of such categories.

\section{Operadic approximation}%
\label{section: approximation}
Because of the adjunction $\algfunctor \dashv \autooperadfunctor$
any functor $F$ admits a \emph{universal operadic approximation}
$\autooperad F$ so that the category of $\autooperad F$-algebras is
universal among those categories of operadic algebras which accept a
restriction functor compatible with $F$ from its domain:
\[
	\begin{tikzcd}
		\domain
		\ar[dr,"F",swap]
		\ar[rr,bend left]
		\ar[r] 
		&
		\alg {\autooperad F}
		\dar
		\rar[dashed,"!"]
		& 
		\alg P
		\ar[dl]
		\\
		& 
		\ground.
	\end{tikzcd}
\]
Here the functor $\domain\to \alg{\autooperad F}$ is the $F$ component
of the counit of the adjunction between algebras and endomorphisms.

In the case that $F$ is right adjoint to a functor $L$, we can
understand the endomorphism operad of $F$ as giving a universal
operadic approximation to the monad $FL$. That is, there is a natural
transformation from the monad associated to the endomorphism operad
$\autooperad F$ to $FL$, and moreover $\autooperad F$ is
terminal among such operads.

\subsection{Loop spaces}
%
%
%

Historically, the first explicit use of the word operad
was in the context of the
recognition principle%
~\cite{doi:10.1007/BFb0067491} which, 
suitably interpreted, identifies an operad that naturally acts
on the $n$-fold loop space functor and then characterizes 
the essential homotopical image of that functor (in, say,
connected spaces) as the category of representations of that operad.
This is a case where the universal operadic approximation recovers
and thus characterizes the functor in question up to homotopy.

\subsection{Groups and sets}

An interesting example is the forgetful functor from groups to
sets. It is well-known that this
forgetful functor is monadic but not operadic (the category of groups
is not a category of algebras over an operad in sets). What is the
operad that describes groups the closest?

\begin{example}%
\label{example: group to set}
Let
\[
	\forget{\Grp} : \Grp \longrightarrow \Sets
\]
be the forgetful functor from groups to sets.
Then $\autooperad {\forget{\Grp}}(n)$ is isomorphic to
the free group on $n$ generators.

Moreover, we can explicitly specify the operadic structure. 
The $\symgroup n$-action is by permutation of generators.
The operadic composition is by word substitution: 
given a free group element $w_1$ on
$x_1,\ldots, x_m$ and a free group element $w_2$ on $y_1,\ldots, y_n$,
then $w_1\circ_i w_2$ is obtained by replacing $x_i$ with $w_2$ and
reindexing.

There are algebras over this operad which are not groups. For example,
any unital \emph{semigroup with involution} admits an action by this
operad. 
So this furnishes an example of a non-trivial monad whose universal
operadic approximation generates a different monad.
\end{example}

\subsection{Singular chains}

We have already discussed the endomorphism operads of the functors of 
cohomology~[\sectionref{subsec:cohomology}] and
differential forms~[\sectionref{subsec: manifolds}].
\begin{question*}
What is the endomorphism operad
of the singular cochain functor
\[
	\mathrm{C}^*:\Top\op\longrightarrow\chains_{\abelian}?
\]
\end{question*}
Of course, this is another context ripe for a homotopical version
of our question. 
Various authors have described~$\einfinity$-operads which act 
functorially on cochains%
~\cite{%
doi:10.1016/S0040-9383(99)00053-1,
doi:10.1017/S0305004102006370,
doi:10.1090/s0894-0347-03-00419-3,
doi:10.1007/s10240-006-0037-6
}. 
By adjunction we get an operad map from any such $\einfinity$-operad 
to the endomorphism operad $\autooperad {\mathrm{C}^*}$ of the cochain 
functor. 
Indeed one of the main results of McClure and Smith is that their 
$\einfinity$-operad is a suboperad of the endomorphism operad of the normalized singular cochain functor.

But we can ask directly: is this map a weak equivalence? 
For some choice of $\einfinity$-operad, is it an isomorphism? 
One of Mandell's results 
is that the lift of $\mathrm{C}^*$ to a functor
\[
	\Top\op\longrightarrow\alg{\einfinity}
\]
is not full (it is homotopically faithful when restricted to finite type
nilpotent spaces). Thus if some version of 
$\einfinity$ is the endomorphism operad of $\mathrm{C^*}$ then 
this is another example where the universal operadic approximation 
fails to yield an equivalence of categories.

\section{Reconstruction of operads}%
\label{section: reconstruction}
Now, following the point of view of duality in the style of Tannaka,
we look at the question of whether we can reconstruct an operad $P$
as the endomorphism operad of the forgetful functor $\forget P$ from
$P$-algebras to $\ground$.

The previous section considered the counit of the $\repfunctor \dashv
\autooperadfunctor$ adjunction.
The question we are considering here is about the unit of the
adjunction, and
in particular whether it is an isomorphism.

\subsection{Reconstruction theorems}
We begin with a variation of a known result about endomorphism monoids.

\begin{restatable}[Reconstruction for monoids]{theorem}
{monoidreconstruction}%
\label{theorem: reconstruction for monoids}
Suppose that $\ground$ is
strongly separated by the monoidal unit.
Then for every monoid $M$, the unit
\[
	\begin{tikzcd}[ampersand replacement=\&]
		M \ar[r]
		\& \automonoid {\forget {M}}
	\end{tikzcd}
\]
of the $\repfunctor \dashv \automonoidfunctor$ adjunction, is an
isomorphism.  In other words, $\repfunctor$ is fully faithful.
\end{restatable}

In the formalism of Tannaka, one considers a variant of the functor
$\repfunctor$ where instead of considering maps $M \to [X,X]$, one uses
enriched maps $[M, [X,X]]$. By doing so, the functor $\repfunctor$ sends
monoids to $\ground$-enriched categories with a $\ground$-enriched
functor to $\ground$.
In such an enriched context, the result of the
theorem holds without a separation axiom for the unit because in
a closed category, the monoidal unit is always a strong
separator in the enriched sense%
~\cite{doi:10.1007/bfb0084235,isbn:978-2856297735}.

In our context, some hypothesis on the ground category is
definitely necessary. For example, the identity functor on
$\integers$-graded $R$-modules can also be viewed as the
forgetful functor for algebras over the trivial $R$-algebra $R$.
If we had reconstruction, then $\automonoidfunctor$ applied to
this identity functor would recover the graded $R$-algebra $R$.
But the endomorphism monoid of this identity functor is
$R^\integers$, not $R$%
~[\ref{lemma:identity on graded R-modules}].

\begin{remark}
Two rings are equivalent in the sense of Morita if their the categories
of representations are equivalent. There are equivalent
rings which are not isomorphic.  So the variant of the functor
$\repfunctor$ that only remembers the category $\rep M$ and not
the forgetful functor $\forget M$ to $\ground$ is not fully
faithful in general. A more refined statement is that an equivalence
in the sense of Morita compatible with forgetful functors must be
induced by an isomorphism of rings.
\end{remark}

We have a similar theorem for operads but with a radically
strengthened hypothesis.

\begin{restatable}[Reconstruction for linear operads]{theorem}
{operadreconstruction}%
\label{thm: infinite field reconstruction}
Let $P$ be an operad in vector spaces over an infinite field $\fieldk$.
Then the unit
\[
	P\longrightarrow \autooperad {\forget P}
\]
is an isomorphism.
\end{restatable}

This fails both for operads in sets and for vector spaces over finite
fields, as shown below.

\subsection{Central operads and failures of reconstruction}
As mentioned above, the endomorphism monoid of the identity functor of
$\ground$ is called the \emph{center} of $\ground$.

In line with this definition we set:
\begin{definition}[Central operad]
\label{definition: central operad}
We shall call the endomorphism operad of the identity functor
\[
	\begin{tikzcd}
		\ground \ar[rr,"\id{}"] && \ground
	\end{tikzcd}
\]
the \emph{central operad} of $\ground$.
\end{definition}

We now turn our attention toward the computation of central operads of
several categories. Each category whose central operad is not the
identity operad is an example of failure of reconstruction.

This simple test already displays interesting behavior and often
obstructs reconstruction. Here are two examples.
\begin{itemize}
	\item The central operad of the category of sets is
	      the $\Perm$ operad;
	\item For $R$ a commutative ring, the central operad of the category
	      of $\integers$-graded $R$-modules is
	      $R^\integers$, concentrated in arity one and acting by scalar
	      multiplication separately in each degree.
\end{itemize}

On the other hand, for ungraded $R$-modules the central operad is the
trivial operad (i.e., $R$ in arity one acting by the module action).
But this is not enough to guarantee reconstruction in general for
$R$-modules, as witnessed by the following example.

\begin{restatable}
{proposition}
{comminpositivecharacteristic}%
\label{prop: comm in positive characteristic}
Let us consider
\[
	\forget \com : \catofalg \com \longrightarrow \vect {\ffield q},
\]
the forgetful functor from $\ffield q$-linear commutative algebras to 
$\ffield q$-vector spaces.

Then the endomorphism operad $\autooperad {\forget \com}$ can be
presented as
\begin{itemize}
	\item generated by a multiplication operation $\mu$ of arity two and a
	      power operation $P_q$ of arity one;
	\item subject to the relations that $\mu$ is commutative and
	      associative and
	      \[
	      	\mu \circ (P_q \otimes P_q) = P_q \circ \mu.
	      \]
\end{itemize}
\end{restatable}

In particular $\autooperad{\forget\com}\ne \com$ so reconstruction fails
for operads in vector spaces over finite fields.

\begin{remark}
The natural action of $P_q$ on a $\com$-algebra is the action of taking
the $q$th power. If $q=p^n$ it may seem as though there should be be a
natural \emph{$p$-th} power action, but the $p$-th power action is not
$\ffield q$-linear.
\end{remark}

\section{Main lemmas}%
\label{section: main lemmas}
Now we turn to the rigorous proofs of the statements we have made in
the earlier sections. Our main tool, introduced in this section, is
a reduction of the size of the cointegral computing the endomorphism
operad (or monoid) which requires a datum and an assumption:
\begin{itemize}
	\item we assume given a separator $S$ for the category $\ground$, and
	\item we assume that the functor $F:\domain\to\ground$ is a right
	      adjoint.
\end{itemize}
We will compute our cointegral over a category built out of the
separator $S$ and the left adjoint to $F$ instead of over all of
$\domain$.

The usefulness of the reduction in computation depends directly on the
size and complexity of $S$ and the objects in it.
In particular, it works well for the category of sets (separated by the
point) and $R$-modules (separated by $R$).

The following definition is standard and is recalled for convenience.
\begin{definition}[Separating set]
A set of objects $S \subset \ground$ is
\emph{separating} if the functor
\[
	\prod_{s \in S} \hom \ground s - : \ground \longrightarrow \Sets
\]
is faithful. 
It is \emph{strongly separating} if moreover it is conservative.
\end{definition}

Now we build the subcategories that we will use to restrict our
cointegral.

\begin{notation}
Let $S \subset \ground$ be a separating set and 
let $F : \domain \to \ground$ admit a left adjoint $L$. 

We write $F_{S^n}$ for the restriction of $F$ 
to the full subcategory of $\domain$ generated by
$L(s_1 \amalg \dots \amalg s_n)$ with $s_i \in S$.
We write $F_S$ for $F_{S^1}$; this is the full subcategory 
of $\domain$ generated by $L(S)$.
\end{notation}

\begin{lemma}[Separator lemma, monoidal case]%
\label{proposition: separators, monoidal case}
Let $S \subset \ground$ be a separating set and 
let $F : \domain \to \ground$ admit a left adjoint $L$. 
Then the
canonical map of endomorphism monoids
\[
	\automonoid F \longrightarrow \automonoid {F_S}
\]
is a monomorphism.
\end{lemma}

\begin{proof}
For every $d \in \domain$, we claim that the map
\[
	[F(d), F(d)] \longrightarrow \prod_{L(s) \to d} [FL(s), F(d)]
\]
is a monomorphism. For this, one can use the unit of the adjunction
$L \dashv F$ and show that the composite
\[
	[F(d), F(d)] \longrightarrow \prod_{L(s) \to d} [FL(s), F(d)]
	\longrightarrow \prod_{L(s) \to d}[s, F(d)]
\]
is a monomorphism. Since $S$ is separating, the canonical map
\[
	\left(\coprod_{L(s) \to d} s\right) \isonat
	\left(\coprod_{s \to F(d)} s\right) \longrightarrow F(d)
\]
is an epimorphism and since the monoidal structure of $\ground$ is
symmetric closed, the functor $X \mapsto [X, F(d)]$ sends 
coproducts to products and epimorphisms to monomorphisms.

Now by the universal property of $\automonoid F$ and $\automonoid
{F_S}$, the following diagram commutes
\[
	\begin{tikzcd}
		\automonoid F \rar \dar[hook]& \automonoid {F_S} \rar
		& \displaystyle\prod_s [FL(s), FL(s)]
		\dar[start anchor={[yshift=2.3ex]}]
		\\
		\displaystyle\prod_d {[F(d),F(d)]} \ar[rr, hook] && \displaystyle
		\prod_{Ls \to d} {[FL(s), F(d)]}
	\end{tikzcd}
\]
which implies that $\automonoid F \to \automonoid {F_S}$ is a
monomorphism.
\end{proof}

\begin{lemma}[Separator lemma, operadic case]%
\label{proposition: separators, operadic case}
Let $S \subset \ground$ be a separating set and 
let $F : \domain \to \ground$ admit a left adjoint $L$. 
Then for every natural
$n$, the canonical map of components of endomorphism operads
\[
	\autooperad F (n) \longrightarrow \autooperad {F_{S^n}}(n)
\]
is a monomorphism.
\end{lemma}%

\begin{proof}
The proof is similar to the previous one. Explicitly, we need to show
that
\[
	[ {F(d)}^{\otimes n}, F(d)]\longrightarrow
	\prod_{\overline s \to F(d)} [{(\overline s)}^{\otimes n}, F(d)
	]
\]
is a monomorphism. As in the previous proof, because $S$ is separating
the canonical map
\[
	\coprod_{s \to F(d)} s \longrightarrow F(d)
\]
is an epimorphism. Since the tensor structure on $\ground$ is closed,
tensorization preserves epimorphisms, so
\[
	\coprod_{\overline s \to F(d)} s_1 \otimes \dots \otimes s_n
	\isonat {\mleft(\coprod_{s \to F(d)} s\mright)}^{\mathclap{\otimes n}}
	{\mkern-4mu}
	\longrightarrow {F(d)}^{\otimes n}
\]
is an epimorphism. Using the factorization
\[
	\begin{tikzcd}
		\displaystyle\coprod_{\overline s \to F(d)} s_1
		\otimes \dots \otimes s_n
		\ar[rr, two heads] \ar[rd,start anchor={[xshift=-7ex,yshift=3ex]}]
		&& {F(d)}^{\otimes n}
		\\
		& \displaystyle\coprod_{\overline s \to F(d)}
		{(\overline s)}^{\otimes n}
		\ar[ru] &
	\end{tikzcd}
\]
we deduce that $\coprod_{\overline s \to F(d)} \overline s^{\otimes n}
\to \power{F(d)}{\otimes n}$ is again an epimorphism.
We end the proof with the fact that $X \mapsto [ X,
F(d)]$ sends coproducts to products and epimorphisms to
monomorphisms.
\end{proof}

\section{Proofs of reconstruction theorems}%
\label{section: proofs of reconstruction}
In this section we use the separator lemmas to prove our two
reconstruction theorems, for monoids in a category with a strong
separator and for operads in the category of vector spaces over an
infinite field.

We restate the first theorem.
\monoidreconstruction*
\begin{proof}
Let $M$ be a monoid in $\ground$.
Since $\ground$ is self enriched, the free representation generated
by $\monu$ is $M$ and we shall denote by
${M}^{\text{full}}$, the full subcategory of
$\rep M$ generated by this single object.
Consider the following composition:
\[
	\begin{tikzcd}
		M \rar
		& [10]\automonoid {\forget M}
		\rar
		& [10] \displaystyle\coint{M^{\text{full}}}
		[\forget M -,\forget M -].
	\end{tikzcd}
\]
We will argue that the composition is an isomorphism (which implies that
the second map is a split epimorphism).
Since the second map is also monic%
~[\ref{proposition: separators, monoidal case}] this will suffice.

To argue that the composition is an isomorphism, first we note
that the unit $M \to \automonoid {\forget M}$ is
monic because self-enrichments are always faithfully tensored over
themselves%
~\cite[4.4]{arXiv:1904.06987}.

Then the separation condition ensures that the underlying functor
\[
	{(-)}_0\coloneqq \hom{\ground}\monu -
	: \ground \longrightarrow \Sets
\]
is conservative. Since it also preserves monomorphisms, what remains to
show is that the map
\[
	\begin{tikzcd}
		M_0 \rar[hook] &
		\displaystyle\coint{M^{\text{full}}}
		\hom{\ground}{\forget M -}{\forget M -}
	\end{tikzcd}
\]
which we now know to be injective, is actually bijective.
Since the monoidal unit is a separator, the transformation
$\hom \ground M M \to \hom \Sets {M_0} {M_0}$ is also injective, thus
we have an injection between the cointegrals
\[
	\begin{tikzcd}
		\displaystyle\coint{M^{\text{full}}}
		\hom{\ground}{\forget M -}{\forget M -}
		\rar[hook] &
		\displaystyle\coint{M^{\text{full}}}
		\hom{\Sets}{{(\forget M -)}_0}
		{{(\forget M -)}_0}.
	\end{tikzcd}
\]
Using the fact that
$\hom {\rep M} M M \isonat
M_0$, one can quickly deduce that this last cointegral is canonically
bijective to $M_0$.

Then the identity of $M_0$ factors as a chain of injective functions:
\[
	\begin{tikzcd}
		M_0 \rar[hook]
		&
		\displaystyle\coint{M^{\text{full}}}
		\hom{\ground}{\forget M -}{\forget M -}
		\rar[hook]
		& M_0
	\end{tikzcd}
\]
inducing the desired bijection.
\end{proof}

In order to prove the operadic reconstruction theorem, we will need a
lemma.

\begin{lemma}%
\label{lemma: gln calculation for reconstruction}
Let $P$ be an operad in vector spaces over an infinite field $\fieldk$
and let $V$ be a vector space of dimension $n$.
Then the natural map $P(n) \otimes V^{\otimes n} \to P \triangletimes V$
induces an $\symgroup n$-equivariant isomorphism
\[
	P(n) \isonat \Hom_{\gl V}(V^{\otimes n},P\triangletimes V).
\]
\end{lemma}

\begin{proof}
One has $\symgroup n$-equivariant isomorphisms
\begin{align*}
	\Hom_{\gl V}(V^{\otimes n},P\triangletimes V)&\isonat
	\Hom_{\gl V}(V^{\otimes n},P(n)\otimes_{\symgroup n}V^{\otimes n}).
	\\
	&
	\isonat
	P(n) \otimes_{\symgroup n} \Hom_{\gl V}(V^{\otimes n},V^{\otimes n}).
	\\
	&
	\isonat P(n).
\end{align*}
The first isomorphism comes by using scalar multiplication by
elements of arbitrarily large order (which exist since $\fieldk$ is
infinite). The second is a consequence of $V$ being finite dimensional.
For the last one, since $V$ has dimension $n$ and $\fieldk$ is
infinite one has
\[
	\Hom_{\gl V}(V^{\otimes n},V^{\otimes n}) \isonat
	\fieldk[\symgroup n]
\]
as an $\symgroup n$-bimodule%
~\cite[A.2.1]{doi:10.1007/978-3-642-30362-3}.
\end{proof}
Now we are ready for the operadic reconstruction theorem.
\operadreconstruction*

\begin{proof}
The proof follows the same general logic as the monoidal case.
Let $V$ be a vector space of dimension $n$ over $\fieldk$, then
the identity of $P(n)$ factors as a chain of monomorphisms as follows.
\begin{enumerate}
	\item The $n$-th component of the unit of the adjunction $P(n)\to
	      \autooperad {\forget P}(n)$ is a monomorphism because $\ground$
	      is faithfully tensored over symmetric sequences in $\ground$%
	      ~(\cite{doi:10.1007/bfb0072514}
	      or%
	      ~\cite[2.3.10]{doi:10.1007/978-3-540-89056-0};%
	      ~\cite[4.4]{arXiv:1904.06987});
	\item Using separators, the $n$-ary component of canonical map of
	      endomorphism operads
	      \[
	      	\autooperad{\forget P}(n)\longrightarrow
	      	\autooperad{\forget P \circ \mathrm{j}_V}(n)
	      \]
	      is a monomorphism%
	      ~[\ref{proposition: separators, operadic case}],
	      where $\mathrm{j}_V$ is the inclusion
	      \[
	      	\begin{tikzcd}
	      	{(\mathrm{Free}_P V)}^\mathrm{full}
	      		\rar[hook, "\mathrm{j}_V"]
	      		&\alg P;
	      	\end{tikzcd}
	      \]
	\item Next, let $(v_1, \dots, v_n)$ be a basis of $V$. Then given any
	      $n$\=/tuple $(a_1, \dots, a_n)$ in $P \triangletimes V$, one can
	      build a $P$-algebra map $f : \text{Free}_P V \to\text{Free}_P V$
	      such that the composite
	      \[
	      	\begin{tikzcd}
	      		V^{\otimes n} \rar["\mathrm{unit}^{\otimes n}"]
	      		&[15] {(\text{Free}_P V)}^{\otimes n} \rar["f^{\otimes n}"]
	      		& {(\text{Free}_P V)}^{\otimes n}
	      	\end{tikzcd}
	      \]
	      sends $v_1 \otimes \dots \otimes v_n$ to
	      $a_1 \otimes \dots \otimes a_n$. This implies that the map
	      induced by restriction along $V \to \forget P(\text{Free}_P V)$
	      	\begin{multline*}
	      	\quad\qquad
	      	\autooperad{\forget P \circ \mathrm{j}_V}(n)
	      	\isonat
	      		\displaystyle\coint{{(\text{Free}_P V)}^{\text{full}}}
	      		\hom{\fieldk}{{(\forget P -)}^{\otimes n}}{\forget P -}
	      		\\
	      	\longrightarrow
	      		\hom{\fieldk}{V^{\otimes n}}{\forget P (\text{Free}_P(V))}
	      	\end{multline*}
	      is a monomorphism.
	      By the universal property of the cointegral, this monomorphism
	      factors through a map 
	      \begin{align*}
	      \autooperad{\forget P \circ \mathrm{j}_V}(n)
	      &\longrightarrow
	      \displaystyle\coint{V^{\text{full}}}
	      		\hom{\fieldk}{{(-)}^{\otimes n}}{P \triangletimes -}
	      	\\&\qquad\eqqcolon
	      	\Hom_{\gl V}(V^{\otimes n},P\triangletimes V) 
	      \end{align*}
	      which is also necessarily a monomorphism. 
	      But this last codomain is isomorphic to 
	      $P(n)$~[\ref{lemma: gln calculation for reconstruction}].
\end{enumerate}
By inspection, the composition is the identity map of $P(n)$, so the
first map in the composition $P(n)\to \autooperad {\forget P}(n)$ is
also an isomorphism.
\end{proof}

As a corollary, one gets the following proposition.

\opoftangent*

\begin{proof}
Let $\simplyconnectedLiegroups$ denote the full subcategory of
simply connected Lie groups. It is a coreflective subcategory of
$\Lie\Grp$
\[
	\begin{tikzcd}[ampersand replacement=\&]
		\simplyconnectedLiegroups \arrow[rr, hook', shift left=2,""]
		\&\&
		\Lie\Grp \arrow[ll, shift left=2, "G \mapsto G^\circ"]
	\end{tikzcd}
\]
where the coreflector $G \mapsto G^\circ$ sends a Lie group to
the universal covering of the connected
component of its unit. Since the counit $G^\circ \to G$ is sent to
an isomorphim by $\tangentate$:
\[
	\tangentate G^\circ = \tangentate G,
\]
the endormorphism operads of the functors
$\tangentate : \Lie\Grp \to \vect\reals$ and $\tangentate :
\simplyconnectedLiegroups \to \vect\reals$ are canonically isomorphic.

Now, $\tangentate$ factors as
\[
	\begin{tikzcd}
		\simplyconnectedLiegroups
		\ar{rr}{=}
		\ar{dr}[swap]{\tangentate}
		&&
		\alg \Lie
		\ar{dl}{\forget\Lie}
		\\
		&
		{\vect\reals}.
	\end{tikzcd}
\]
Then the reconstruction theorem above
shows that $\autooperad {\forget\Lie}$ is $\Lie$, hence
$\autooperad \tangentate$ is also the Lie operad.
\end{proof}

\section{Computations}%
\label{section: computations}

In this section we tie up remaining loose ends, giving the details of
deferred computations. In each case we shall use the same stategy:
given a right adjoint $F : \mathcal D \to \ground$,
\begin{enumerate}
	\item select a separator $S \subset \ground$;
	\item use the operadic separator lemma%
~[\ref{proposition: separators, operadic case}] to get monomorphisms
	      \[
	      	\begin{tikzcd}
	      		\autooperad F (n) \ar[rr, hook] && \autooperad {F_{S^n}}
	      		(n);
	      	\end{tikzcd}
	      \]
	\item identify the biggest subobject of
	      $\autooperad {F_{S^n}} (n)$ that acts naturally on $F$
	      compatibly with the action on $F_{S^n}$; this is
	      $\autooperad F(n) $.
\end{enumerate}

The separators that we shall use are the typical ones: the singleton
for sets, $R$ in the case of modules over a ring $R$ and the set
$\under{\{R[n]\}}{n\in \integers}$ in the case of graded $R$-modules.

\begin{lemma}%
\label{lemma: group to set}
Let
\[
	\forget \Grp : \Grp \longrightarrow \Sets
\]
be the forgetful functor from groups to sets.
Then $\autooperad {\forget \Grp}(n)$ is naturally isomorphic to the 
set underling a free group on $n$ generators, with action as 
specified below.
\end{lemma}
\begin{proof}
Let $\freegroup X$ be the free group on $X = \{x_1,\ldots,x_n\}$.
We shall show that
\[
	\coint{{\freegroup X}^{\text{full}}}\hom {\Sets}
	{\forget \Grp -^{n}} {\forget \Grp -}
\]
is naturally isomorphic to $\freegroup X$.

Let $\phi$ be a map from $\power{\freegroup X} n$ to $\freegroup X$
commuting with every group homomorphism of $\freegroup X$,
let $g_1,\ldots, g_n$ be a tuple in $\power{\freegroup X} n$,
and let $f$ be the group homomorphism taking $x_i$ to $g_i$.
Then 
\begin{align*}
	\phi(g_1,\ldots, g_n)&=\phi(f(x_1),\ldots f(x_n))
	\\
	&= f(\phi(x_1,\ldots,x_n))
\end{align*}
so $\phi$ is determined by its value on $(x_1,\ldots,x_n)$, which is an
element of $\freegroup X$.

These are all distinct because they take different values on the tuple
$(x_1,\ldots, x_n)$, so it remains only to argue that all of these in
fact yield natural maps for arbitrary groups.

For any group $G$, an element in $\freegroup X$ yields a map $G^{\times
n}\to G$ which takes $(g_1,\ldots, g_n)$ to the word in $G$ obtained
from $w$ by replacing $x_i$ with $g_i$; this is clearly natural.
\end{proof}

For computations of endomorphisms of identity functors, recall that the
identity functor is isomorphic to the forgetful functor from algebras
over the trivial operad.

\begin{proposition}%
The central operad of the category $\catofmod R$ of modules over
a commutative ring $R$ is the trivial operad (i.e., $R$ in arity one).
\end{proposition}%

\begin{proof}
We shall compute the cointegrals
\[
	\coint{{(R^n)}^{\text{full}}}\hom R {-^{\otimes n}} -
\]
i.e., the linear maps $\phi$ from $\power{(R^n)}{\otimes n}$ to $R^n$
commuting with every map $R^n\to R^n$.

\begin{itemize}
	\item For $n=0$, we have $R^n=R^0=0$,
	      and there is only one map $R\to 0$;
	\item For $n=1$, the center of $\gl 1(R)$ is $R$ itself;
	\item For $n>1$, commuting with the map $\pi_i:R^n\to R^n$ killing the
	      basis element $e_i$ and acting as the identity on the other
	      variables means that any such $\phi$ must take
	      $e_1\otimes\cdots\otimes e_n$ to the kernel of $\pi_i$.  For
	      distinct $i$ and $j$, the kernel of $\pi_i$ and the kernel of
	      $\pi_j$ have null intersection so any such transformation must
	      take the generic primitive tensor to zero.
\end{itemize}

Thus this is the trivial operad $R$, which clearly acts naturally by
scalar multiplication.
\end{proof}

\begin{proposition}%
\label{lemma:identity on graded R-modules}
The central operad of the category of $\integers$-graded $R$-modules
is $R^\integers$, concentrated in arity one
(it acts by scalar multiplication separately in each degree).
\end{proposition}

\begin{proof}
We shall compute a cointegral over modules of the form 
\[
	R^{\vec{\jmath}}=R[j_1]\oplus \cdots \oplus R[j_n].
\]
Then we need to determine the maps
$\power{(R^{\vec{\jmath}})}{\otimes n} \to R^{\vec{\jmath}}$ which
commute with every module map from $R^{\vec{\jmath}}$
to itself. The computation in the ungraded case goes through as
before for each $\vec{\jmath}$. In particular, this shows that
$\autooperad {\id {}}$ is concentrated in arity one. For the calculation
in arity one, the full subcategory spanned by the modules $R[j]$ has no
maps between $R[j_1]$ and $R[j_2]$ for distinct $j_1$ and $j_2$. Thus
\[
	\coint{{\{{R[j]}_{j\in \integers}\}}^{\text{full}}}[-,-]
	\isonat
	\prod_{j\in \integers}\int_{{R[j]}^{\text{full}}}^*[-,-]
	\isonat
	\prod_{j\in \integers}R.
\]
On the other hand, we can realize such a product $\under{(r_j)}{j\in
\integers}$ as the natural transformation which multiplies degree $j$ by
$r_j$.
\end{proof}

The situation is more interesting in sets. Let us recall that
$\Perm$ is the operad whose algebras are sets endowed with
a binary operation $(x,y) \mapsto xy$ satisfying
$(xy)z = x(yz) = x(zy)$.

\begin{proposition}%
\label{lemma:identity on sets}
The central operad of the category of sets is the $\Perm$ operad.
\end{proposition}%

\begin{proof}
We shall compute
\[
	\coint{{[n]}^{\text{full}}}\hom {\Sets} {-^{n}} -
\]
i.e., the maps $\phi$ from $\power{[n]}{n}$ to $[n]$ commuting with
every endomorphism of $[n]$, where $[n]$ denotes the set $\{1, \dots,
n\}$. For $n=0$ this set is empty. For $n$ strictly larger than zero,
let $\phi$ be such a function and suppose $\phi(1,\ldots, n)=i$. For any
tuple $(a_1,\ldots, a_n)$ be a tuple in $\power{[n]}{n}$ let $f$ be the
function $[n]\to [n]$ which takes $j$ to $a_j$. Then
\begin{align*}
	\phi(a_1,\ldots,a_n)&=\phi(f(1),\ldots, f(n))
	\\&=f(\phi(1,\ldots, n))
	\\&=f(i)
\end{align*}
which shows that $\phi$ is the projection in the $i$th coordinate.

Thus this $n$-th cointegral is the set
$\{\pi_1,\ldots,\pi_n\}$, with symmetric group
action on the subscript.
There are evident relations between the projections for $n=2$:
\begin{align*}
	\pi_1 \circ_1 \pi_1 = \pi_1\circ_2\pi_1 = \pi_1 \circ_2 \pi_2 
\end{align*}
as all of these are the projection $\pi_1$ in arity $3$.
This is a presentation for the $\Perm$ operad which acts naturally
on sets in the obvious way.
\end{proof}

Finally we conduct our computation of the endomorphism operad of the
forgetful functor from commutative algebras to vector spaces over a
finite field.

\comminpositivecharacteristic*
\begin{proof}
Let us consider the cointegral
\[
	\coint{\sym{V_n}^\mathrm{full}} \hom {\ffield q} {-^{\otimes n}} -
\]
where $\sym{-}$ denotes the symmetric algebra 
and $V_n$ is an $n$-dimensional vector space with a choice of basis.

Let $(e_1, \dots, e_n)$ be a basis of $V_n$ and let $(f_1,\ldots,
f_n)$ be an ordered set of polynomials in the variables
$\mathrm{X}_1,\ldots,\mathrm{X}_n$. There is an endomorphism of
$\sym{V_n}$ which takes $e_i$ to $f_i(e_1,\ldots, e_n)$. Equalizing
over this algebra map implies that any function in the cointegral is
fully determined by its value $f$ on $(e_1,\dots, e_n)$, so that there
is a monomorphism
\[
	\begin{tikzcd}
		\displaystyle
		\coint{\sym{V_n}^\mathrm{full}} \hom {\ffield q} {-^{\otimes n}} -
		\ar[rr, hook] && \sym{V_n}.
	\end{tikzcd}
\]

We are now left to find the biggest subobject of $\sym{V_n}$ that acts
naturally on $\forget\com$.

Thus consider an element of $\sym{V_n}$, i.e., a polynomial of the form
\[
	f=\sum_{m_1,\ldots, m_n} \alpha_{\vec{m}}\mathrm{X}_1^{m_1}\cdots
	\mathrm{X}_n^{m_n}.
\]
Such a polynomial acts set-theoretically on commutative algebras $A$ 
via evaluation
\[
	\mathrm{Ev}_f (A): A^{\times n} \longrightarrow A.
\]
This is already natural as a map of \emph{sets}.
But we must restrict to those $f$ for which $\mathrm{Ev}_f$ is
a $\ffield q$-multilinear, i.e., to induce a map
\[
A^{\otimes n}\longrightarrow A.
\]

For this, let us see what multilinearity means in the case where
$A=\sym{V_{n+1}}$ is the free commutative algebra on $n+1$
generators $e_1, \dots, e_{n+1}$: consider the equation in
$\sym{V_{n+1}}^{\times n}$
\begin{multline*}
	\label{eq: sum in sym}
	(e_1\dots e_n) + 
	(e_1\dots e_{i-1}, e_{n+1},
	e_{i+1}\dots e_n)
	\\
	=
	(e_1\dots e_{i-1},(e_i+e_{n+1}),
	e_{i+1}\dots e_n).
\end{multline*}
Via direct computation, a first condition for the function
$\mathrm{Ev}_f(\sym{V_{n+1}})$ to be linear with respect to the above
equation is that each of the monomials of $f$ act linearly in its
$i$-th variable.

Then for the linearity of each monomial, the above equation yields
also
\[
	\power{(e_i+e_{n+1})}{m_i}=e_i^{m_i}+e_{n+1}^{m_i}.
\]
This is only possible if $m_i$ is a power of $p$, the characteristic of
$\ffield q$. This already implies, e.g., that $m_i\ne 0$.

For multilinearity we also need $\power{(\alpha e_i)}{m_i}$ to equal
$\alpha e_i^{m_i}$ for arbitrary $\alpha\in \ffield q$. This implies
that $q-1$ must divide $m_i-1$. This fact then further implies that
$m_i$ is not only a power of $p$ but also a power of $q$.

So far, we have argued that the $n$-ary operations in
$\autooperad{\forget\com}$ inject into the $\ffield q$ polynomials in
$\mathrm{X}_1,\ldots, \mathrm{X}_n$ such that the exponent of each
$\mathrm{X}_i$ is a power of $q$. Conversely any such polynomial
clearly acts linearly. 
This presentation corresponds to the presentation in the hypotheses of
the proposition as follows. The binary product $\mu$ corresponds to the
two-variable monomial $\mathrm{X}_1\mathrm{X}_2$ and the power operation
$P_q$ corresponds to the one-variable monomial $\mathrm{X}_1^q$.
Operadic composition is via substitution of variables.
\end{proof}

\section*{Acknowledgments}

The authors would like to thank Rune Haugseng, Theo Johnson-Freyd,
Claudia Scheimbauer, and John Terilla for useful discussions, Greg
Arone, Michael Batanin and Birgit Richter for pointing us to pertinent
references in the literature, and Benoit Fresse for a remark that
partially inspired this paper.

\bibliography{bib/dl}
\bibliographystyle{sty/dl-en}
\end{document}

%% file: config.tex
\usepackage{thm-restate}

\title{Endomorphism operads of functors}

\author{Gabriel C. Drummond-Cole}
\thanks{The first and third authors were supported by IBS-R003-D1.}
\address{Center for Geometry and Physics, Institute for Basic Science
	(IBS), Pohang, Republic of Korea 37673}
\email{gabriel.c.drummond.cole@gmail.com}

\author{Joseph Hirsh}
\thanks{The second author was supported by NSF DMS-1304169}
\email{josephhirsh@gmail.com}

\author{Damien Lejay}
\address{Center for Geometry and Physics, Institute for Basic Science
	(IBS), Pohang, Republic of Korea 37673}
\email{lejay@paracompact.space}

\date{}

\hypersetup{%
	pdftitle={Endomorphisms of functors},
	pdfauthor={\textcopyright%
		Gabriel C. Drummond-Cole, Joseph Hirsh and Damien Lejay},
	pdfsubject={},
	pdfkeywords={},
	pdfcreator={pdfLaTeX},
	pdfproducer={LaTeX with hyperref}
}

\theoremstyle{plain}

\theoremstyle{definition}

\newtheorem{notation}[theorem]{Notation}

\newtheorem*{question*}{Question}

%% file: macros.tex
\newcommand{\ground}{\mathscr{C}}
\newcommand{\domain}{\mathcal{D}}

\newcommand{\Perm}{\mathbf{Perm}}
\newcommand{\Sets}{\mathsf{Set}}
\newcommand{\Grp}{\mathsf{Grp}}
\newcommand{\Top}{\mathsf{Top}}

\newcommand{\slice}[2]{#1_{/#2}}
\newcommand{\vect}[1]{\mathsf{Vect}_{#1}}

\DeclareFontFamily{U}{min}{}
\DeclareFontShape{U}{min}{m}{n}{<-> udmj30}{}
\newcommand{\sectionref}[1]{\hyperref[#1]{\S\thinspace\ref{#1}}}
\newcommand{\morita}[1]{\mathsf{Mor}_*}

\newcommand{\gl}[1]{\operatorname{GL}_{#1}}

\newcommand{\com}{\mathsf{C}}
\newcommand{\fieldk}{\boldsymbol{K}}
\newcommand{\ffield}[1]{\mathbf{F}_{#1}}

\newcommand{\coint}[1]{\int_{#1}^{\ast}}
\newcommand{\forget}[1]{\mathrm{U}_{#1}}

\newcommand{\symgroup}[1]{\mathbf{S}_{#1}}
\newcommand{\triangletimes}{\mathbin{\triangleleft}}

\newcommand{\isonat}{=}

\newcommand{\under}[2]{#1_{#2}}
\newcommand{\power}[2]{#1^{#2}}

\newcommand{\sym}[1]{\mathrm{S}(#1)}

\newcommand{\auto}[1]{\mathscr{E}(#1)}
\newcommand{\automonoid}[1]{\mathscr{E}_1(#1)}
\newcommand{\autooperad}[1]{\auto{#1}}
\newcommand{\autofunctor}{\mathscr{E}}
\newcommand{\automonoidfunctor}{\autofunctor_1}
\newcommand{\autooperadfunctor}{\autofunctor}

\newcommand{\alg}[1]{#1\text{-}\mathsf{alg}}
\newcommand{\rep}[1]{#1\text{-}\mathsf{rep}}

\newcommand{\repfunctor}{\mathsf{Rep}}
\newcommand{\algfunctor}{\mathsf{Alg}}

\newcommand{\accesscats}{\mathsf{Acc}}

\newcommand{\monu}{\mathbf{1}}

\newcommand{\monoids}[1]{\mathsf{Mon}\mleft(#1\mright)}
\newcommand{\operads}[1]{\mathsf{Op}\mleft(#1\mright)}

\newcommand{\abelian}{\mathsf{Ab}}
\newcommand{\chains}{\mathsf{Ch}}

\renewcommand{\op}{^\mathrm{op}}
\renewcommand{\hom}[3]{\mathrm{Hom}_{#1}(#2, #3)}
\newcommand{\nat}[3]{\mathrm{Nat}_{#1}(#2, #3)}
\newcommand{\catofmod}[1]{#1\text{-}\mathsf{mod}}
\newcommand{\catofalg}[1]{#1\text{-}\mathsf{alg}}

\newcommand{\freegroup}[1]{\operatorname{F}_{#1}}
\newcommand{\einfinity}{\mathsf{E}_\infty}
\newcommand{\simplyconnectedLiegroups}{\mathsf{LieGrp}^\circ}
\newcommand{\Lie}{\mathsf{Lie}}
\newcommand{\tangentate}{\mathrm{T}_\mathrm{e}}

\DeclareMathOperator{\Hom}{Hom}